\documentclass[a4paper,11pt]{article}

\usepackage[T1]{fontenc}
\usepackage{lmodern}

\usepackage{dsfont}

\usepackage[latin1]{inputenc}
\usepackage{amsmath}
\usepackage{amsthm}
\usepackage{amssymb}
\usepackage{mathrsfs}
\usepackage{graphicx}
\usepackage[all]{xy}
\usepackage{hyperref}

\usepackage{makeidx}

\usepackage{stmaryrd}
\usepackage{caption}

\usepackage{abstract} 

\newtheorem{thm}{Theorem}[section]
\newtheorem{cor}[thm]{Corollary}

\newtheorem{lemma}[thm]{Lemma}

\theoremstyle{definition}
\newtheorem{definition}[thm]{Definition}
\newtheorem{ex}[thm]{Example}
\newtheorem{remark}[thm]{Remark}
\newtheorem{question}[thm]{Question}

\def\rquotient#1#2{%
	\makeatletter
	\raise.3ex\hbox{$#1$}/\lower.3ex\hbox{$#2$}%
	\makeatother
}	

\makeatletter
\newcommand{\subjclass}[2][2010]{%
	\let\@oldtitle\@title%
	\gdef\@title{\@oldtitle\footnotetext{#1 \emph{Mathematics subject classification.} #2}}%
}
\newcommand{\keywords}[1]{%
	\let\@@oldtitle\@title%
	\gdef\@title{\@@oldtitle\footnotetext{\emph{Key words and phrases.} #1.}}%
}
\makeatother

\newcommand{\Address}{{
		\bigskip
		\small
		
\noindent
		\textsc{Laboratoire de Math\'ematiques Nicolas Oresme UMR CNRS 6139\\ Universit\' e de Caen Normandie \\ 14000 Caen, France}\par\nopagebreak
		\textit{E-mail address}: \texttt{paolo.bellingeri@unicaen.fr}

\medskip \noindent
		\textsc{Universit\'e de Montpellier\\ 
Institut Math\'ematiques Alexander Grothendieck\\
Place Eug\`ene Bataillon\\
34090 Montpellier (France)}\par\nopagebreak
		\textit{E-mail address}: \texttt{anthony.genevois@umontpellier.fr}

\medskip \noindent
		\textsc{Laboratoire de Math\'ematiques Nicolas Oresme UMR CNRS 6139\\ Universit\' e de Caen Normandie \\ 14000 Caen, France}\par\nopagebreak
		\textit{E-mail address}: \texttt{neha.nanda@unicaen.fr}
		
}}

\makeindex

\title{Right-angled Artin groups are symmetric diagram groups}
\date{\today}
\author{Paolo Bellingeri, Anthony Genevois, and Neha Nanda}


\begin{document}
\maketitle

\begin{abstract}
In this article, we show that, for every $n \geq 2$, the pure virtual twin group $PVT_n$ can be naturally described as a symmetric diagram group, a family of groups introduced by V. Guba and M. Sapir and associated to semigroup presentations. Inspired by this observation, we prove that every finitely generated right-angled Artin group is a symmetric diagram group. This contrasts with the fact that not all right-angled Artin groups are planar diagram groups.
\end{abstract}

\section{Introduction}

\noindent
Guba-Sapir extended in \cite{MR1396957} combinatorial group theory to  the realm of diagram groups. One of the most motivating examples was Thompson's group $F$, the group of all piecewise-linear orientation-preserving homeomorphisms of the interval whose break points are dyadic integers, and whose values in the break points are also dyadic integers. The study of diagram groups lead to new results about Thompson's groups $F$. Since then, these groups have been studied substantially, in particular, towards investigating properties of other groups which can be viewed as diagram groups or their subgroups. Framework on this subject includes computational problems such as the word, conjugacy and
commutation problems, finiteness properties, homology,
orderability and different  geometries, to name a few. We refer to the recent survey article \cite{Diagram-Groups-Genevois} for a detailed account.\\
Interestingly, diagram groups have dissimilar yet equivalent interpretations. One of these alternative definitions allows us to define different  generalisations of diagram groups. For distinction, classical diagram groups, from now on referred as \textit{planar diagram groups}, have
their elements represented as pictures on the  plane consisting of non-intersecting wires and transistors. A similar definition on annuli yields \textit{annular diagram group}, and
allowing wires to cross in an arbitrary manner yields \textit{symmetric diagram groups} \cite{MR1396957, MR2136028}. These extensions mimic the case of Thompson's groups $F \subset T \subset V$. That is, for every semigroup presentation $\mathcal{P}= \langle \Sigma \mid \mathcal{R} \rangle$ and every non-empty word $w \in \Sigma$, we define three nested groups $D_p(\mathcal{P},w) \subset D_a(\mathcal{P},w) \subset D_s(\mathcal{P},w)$, where $D_p(\mathcal{P},w)$ (resp. $D_a(\mathcal{P},w)$, $D_s(\mathcal{P},w)$) is a planar (resp. annular, symmetric) diagram group. \\
The characterisation of groups which are (planar) diagram groups is widely open. And, usually, it is difficult to determine whether or not a given group can be described as a diagram group. The particular case of right-angled Artin groups is interesting. It is known that not every right-angled Artin group is a planar diagram group \cite{MR1725439}, but a complete classification is unknown. This motivates us to investigate the relation between right-angled Artin groups and symmetric diagram groups. In the symmetric case, we prove that every right-angled Artin group turns out to be a symmetric diagram group.

\begin{thm}\label{thm:MainIntro}
Every finitely generated right-angled Artin group is a symmetric diagram group.
\end{thm}

\noindent
As a nice application, it follows that there exist torsion-free symmetric diagram groups that are not planar diagram groups. It is clear that symmetric diagram groups may not be planar diagram groups, since the latter are torsion-free while the former may contain torsion. Examples include for instance Thompson's group $V$ or Houghton groups. However, whether torsion-free symmetric diagram groups are automatically planar diagram groups is not clear at first glance. Up to our knowledge, we exhibit the first example providing a negative answer.

\begin{cor}\label{Symmetric-Not-Planar}
There exist torsion-free symmetric diagram groups, such as right-angled Artin groups defined by cycles of odd length $\geq 5$, that are not planar diagram groups.
\end{cor}

The fact that the right-angled Artin groups mentioned above are not planar diagram groups is proved in \cite{MR1725439}. See also Theorem~\ref{thm:NotPlanar}.

\medskip
Besides, this article explores the world of symmetric diagram groups for which not much is known. In this direction, we  consider in particular,  twin groups and their generalisations, and relate their pure subgroups to  planar, annular, and symmetric diagram groups. The term \textit{twin group} was coined in the work of Khovanov \cite{MR1370644}, and these groups appear in the literature under the names of Grothendieck cartographical groups \cite{MR1106916}, planar braid groups \cite{MR4170471}, traid groups \cite{MR4035955} and flat braids \cite{MR1738392}.   They are right-angled Coxeter groups which serve as planar analogues of classical Artin braid groups.  These groups relate to the doodles on $2$-sphere which are finite collections of closed curves that are required not to have triple intersections. The kernel of a natural map from twin group onto the symmetric group is known as \textit{pure twin group} $PT_n$ which is the fundamental group of real codimension two subspace arrangements. In other words, it is the fundamental group of configuration space of $n$ particles on a line such that no three particles can collide. Through \cite{Twin-Group-Diagram-Group,article3}, it is known that $PT_n$ is a planar diagram group.\\
From the motivation of virtual knots and braids, arose the virtual extension of twin groups \cite{MR4027588}. Similar to the pure twin group, the canonical subgroup of virtual twin group is called \textit{pure virtual twin group} $PVT_n$ which turns out to be  an irreducible right-angled Artin group \cite{Structure-Automorphisms-Pure-Virtual-Twin-Groups}. From the before mentioned results, we deduce that the pure virtual twin group is a symmetric diagram group. 
However, the information of the elements of these groups can be translated in terms of monotonic strings, real and virtual crossings on the plane. This allows us to give an alternative and clearer proof where the symmetric aspect of diagram group is evidently visible. We prove the following result. 

\begin{thm}\label{Main-Theorem-Pure-Virtual-Twin}
Let $\mathcal{P}_n= \langle x_1, x_2, \dots, x_n ~|~ x_i x_j =x_j x_i \text{ for all } i \neq j \rangle$ be a semigroup presentation. The symmetric diagram group $D_s(\mathcal{P}_n,x_1x_2 \cdots x_n )$ is isomorphic to the pure virtual twin group $PVT_n$. 
\end{thm}
For completion, we also deduce from \cite{MR1725439} that the right-angled Artin group $PVT_n$ is not a planar diagram group, \ref{Symmetric-Not-Planar} providing another example illustrating Corollary~\ref{Symmetric-Not-Planar}. Quite recently, Mostovoy defined \textit{annular twin groups} which is motivated by considering configuration space of $n$ particles on a circle than a line, whose fundamental group corresponds to \textit{pure annular twin group} $aPT_n$. Farley \cite[Theorem 2]{Twin-Group-Diagram-Group} independently considered and defined pure annular twin groups in the sense of diagrams, and proved that $aPT_n$ is an annular diagram group. As an immediate corollary to Theorem \ref{Main-Theorem-Pure-Virtual-Twin}, we get the following non-trivial result.
\begin{cor}
The pure twin group embeds in the pure annular twin group, which further embeds in the pure virtual twin group.
\end{cor}
It is worth mentioning that it is not a trivial question to ask if a group embeds in its virtual extension. For example, we can prove  the braid groups embed in virtual braid groups (and  similarly that $T_n$ embeds  in $VT_n$) using classical result  on embeddings between parabolic subgroups of Artin and Coxeter groups (see for instance \cite{MR3519103}).

In summary, symmetric diagrams   can be roughly viewed as a virtual extension of planar diagrams, in the sense of knots and virtual knots,  braids and virtual braids, twin groups and virtual twin groups. More specifically,  as virtual (pure) braids provide all possible crossings between different strands, symmetric diagrams realize all possible combinations of wires. Recently in \cite{Virtual-Artin-Groups}, it was pointed out that virtual braid groups can be also realised mimicking the action of the symmetric group on its root system. An interesting consequence is that we can obtain this way  the usual presentation of pure virtual braid groups, where generators are in one to one correspondence with simple roots of the symmetric group.  It seems interesting
to investigate if symmetric diagrams can be similarly defined in terms of underlying 
actions of the symmetric group on the geometry of the corresponding planar diagram group.

\medskip

\textbf{Acknowledgements.} First author was partially supported by the ANR project AlMaRe (ANR-19-CE40-0001).The third author has received funding from European Union's Horizon Europe Research and Innovation programme under the Marie Sklodowska Curie grant agreement no. 101066588.

\section{Diagram groups and diagram products}

\subsection{Diagram groups}\label{section:diagrams}

In literature, planar diagram groups have various distinct, yet equivalent definitions. Informally, they are either defined as fundamental group of Squier square complex, or in terms of groups consisting of   equivalence classes of certain pictures comprising of labelled directed graphs with cells (which leads to viewing planar diagram groups as two-dimensional analogues of free groups). They can also be viewed as collection of equivalent dual pictures consisting of transistors and wires.  This last will be the  formal definition of which we will adapt throughout the paper. 
\medskip
\paragraph{\textbf{Symmetric semigroup diagrams.}} In this paragraph, we introduce the main definitions related to \emph{symmetric semigroup diagrams}, firstly introduced in \cite{MR1396957} and further studied in \cite{MR2136028}. We follow the definition given in \cite{MR3822286} and we refer to Example \ref{ex:diagram} below to get a good picture to keep in mind when reading the definition (notice, however, that on our diagrams the frames are not drawn). 

\medskip \noindent
The \textit{free semigroup} on the set $\Sigma$ is the set $\Sigma^+$ of all non-empty strings formed from $\Sigma$, that is, 
$$\Sigma^{+}= \{ u_1u_2\dots u_n ~|~ n \in \mathbb{N}, ~u_i \in \Sigma \text{ for all }i\}.$$
Let $1$ denotes the empty string and $\Sigma^*= \Sigma^+ \cup \{1\}$. We define operation of concatenation on sets $\Sigma^+$ and $\Sigma^*$. We say $$w_1 = w_2$$ if $w_1$ and $w_2$ are equal as words in $\Sigma^*$. A \emph{set of relations} $\mathcal{R}$ is a subset of $\Sigma^+ \times \Sigma^+$. For clarity, we denote an element $(u,v)$ of $\mathcal{R}$ as $u=v$.\\
A \textit{semigroup presentation} $\mathcal{P}:= \langle \Sigma \mid \mathcal{R} \rangle$ is the data of an alphabet $\Sigma$ and a set of relations $\mathcal{R}$. In all our paper, we follow the convention that, if $u=v$ is a relation of $\mathcal{R}$, then $v=u$ does not belong to $\mathcal{R}$; as a consequence, $\mathcal{R}$ does not contain relations of the form $u=u$.\\
\medskip

We now begin to define the fundamental pieces which we will glue together to construct (symmetric) diagrams.
\begin{itemize}
	\item A \emph{wire} is a homeomorphic copy of $[0,1]$. The point $0$ is the \emph{bottom} of the wire, and the point $1$ its \emph{top}.
	\item A \emph{transistor} is a homeomorphic copy of $[0,1]^2$. One says that $[0,1] \times \{1 \}$ (resp. $[0,1] \times \{ 0 \}$, $\{0\} \times [0,1]$, $\{1 \} \times [0,1]$) is the \emph{top} (resp. \emph{bottom}, \emph{left}, \emph{right}) \emph{side} of the transistor. Its top and bottom sides are naturally endowed with left-to-right orderings.
	\item A \emph{frame} is a homeomorphic copy of $\partial [0,1]^2$. It has \emph{top}, \emph{bottom}, \emph{left} and \emph{right sides}, just as a transistor does, and its top and bottom sides are naturally endowed with left-to-right orderings.
\end{itemize}
Fix a semigroup presentation $\mathcal{P}= \langle \Sigma \mid \mathcal{R} \rangle$. A \emph{symmetric diagram over $\mathcal{P}$} is a labelled oriented quotient space obtained from
\begin{itemize}
	\item a finite non-empty collection $W(\Delta)$ of wires;
	\item a labelling function $\ell : W(\Delta) \to \Sigma$;
	\item a finite (possibly empty) collection $T(\Delta)$ of transistors;
	\item and a frame,
\end{itemize}
which satisfies the following conditions:
\begin{itemize}
	\item each endpoint of each wire is attached either to a transistor or to the frame;
	\item the bottom of a wire is attached either to the top of a transistor or to the bottom of the frame;
	\item the top of a wire is attached either to the bottom of a transistor or the top of the frame;
	\item the images of two wires in the quotient must be disjoint;
	\item if $\mathrm{top}(T)$ (resp. $\mathrm{bot}(T)$) denotes the word obtained by reading from left to right the labels of the wires which are connected to the top side (resp. the bottom side) of a given transistor $T$, then either $\mathrm{top}(T)= \mathrm{bot}(T)$ or $\mathrm{bot}(T)= \mathrm{top}(T)$ belongs to $\mathcal{R}$;
	\item if, given two transistors $T_1$ and $T_2$, we write $T_1 \prec T_2$ when there exists a wire whose bottom contact is a point on the top side of $T_1$ and whose top contact is a point on the bottom side of $T_2$, and if we denote by $<$ the transitive closure of $\prec$, then $<$ is a strict partial order on the set of the transistors. 
\end{itemize}
Two symmetric diagrams are \emph{equivalent} if there exists a homeomorphism between them which preserves the labellings of wires and all the orientations (left-right and top-bottom) on all the transistors and on the frame. From now on, every symmetric diagram will be considered up to equivalence. This allows us to represent a diagram in the plane so that the frame and the transistors are straight rectangles such that their left-right and top-bottom orientations coincide with a fixed orientation of the plane, and so that wires are vertically monotone. 

\medskip \noindent
Given a symmetric diagram $\Delta$, one defines its \emph{top label} (resp. its \emph{bottom label}), denoted by $\mathrm{top}(\Delta)$ (resp. $\mathrm{bot}(\Delta)$), as the word obtained by reading from left to right the labels of the wires connected to the top (resp. the bottom) of the frame. A \emph{symmetric $(u,v)$-diagram} is a symmetric diagram whose top label is $u$ and whose bottom label is $v$. 

\medskip \noindent
Fixing two symmetric diagrams $\Delta_1$ and $\Delta_2$ satisfying $\mathrm{top}(\Delta_2)= \mathrm{bot}(\Delta_1)$, one can define their \emph{concatenation} $\Delta_1 \circ \Delta_2$ by gluing bottom endpoints of the wires of $\Delta_1$ which are connected to the bottom side of the frame to the top endpoints of the wires of $\Delta_2$ which are connected to the top side of the frame, following the left-to-right ordering, and by identifying, and next removing, the bottom side of the frame of $\Delta_1$ with the top side of the frame of $\Delta_2$. Loosely speaking, we ``glue'' $\Delta_2$ below $\Delta_1$.

\medskip \noindent
Given a symmetric diagram $\Delta$, a \emph{dipole} in $\Delta$ is the data of two transistors $T_1,T_2$ satisfying $T_1 \prec T_2$ such that, if $w_1, \ldots, w_n$ denotes the wires connected to the top side of $T_1$, listed from left to right:
\begin{itemize}
	\item the top endpoints of the $w_i$'s are connected to the bottom of $T_2$ with the same left-to-right order, and no other wires are attached to the bottom of $T_2$;
	\item the top label of $T_2$ is the same as the bottom label of $T_1$. 
\end{itemize}
Given such a dipole, one may \emph{reduce} it by removing the transistors $T_1, T_2$ and the wires $w_1, \ldots, w_n$, and connecting the top endpoints of the wires which are connected to the top side of $T_2$ with the bottom endpoints of the wires which are connected to the bottom side of $T_1$ (preserving the left-to-right orderings). 

\medskip \noindent
A symmetric diagram without dipoles is \emph{reduced}. Clearly, any symmetric diagram can be transformed into a reduced one by reducing its dipoles, and according to \cite[Lemma~2.2]{MR2136028}, the reduced diagram we get does not depend on the order we choose to reduce its dipoles. We refer to this reduced diagram as the \emph{reduction} of the initial diagram. Two diagrams are the same \emph{modulo dipoles} if they have the same reduction. 
\begin{ex}\label{ex:diagram}
Consider the semigroup presentation 
$$\mathcal{P}= \langle a,b,c \mid ab=ba, ac=ca, bc=cb \rangle.$$
Figure \ref{figure9} shows the concatenation of two symmetric diagrams over $\mathcal{P}$, and the reduction of the resulting symmetric diagram. 
\begin{figure}
\begin{center}
\includegraphics[trim={0 11cm 20cm 0},clip,scale=0.35]{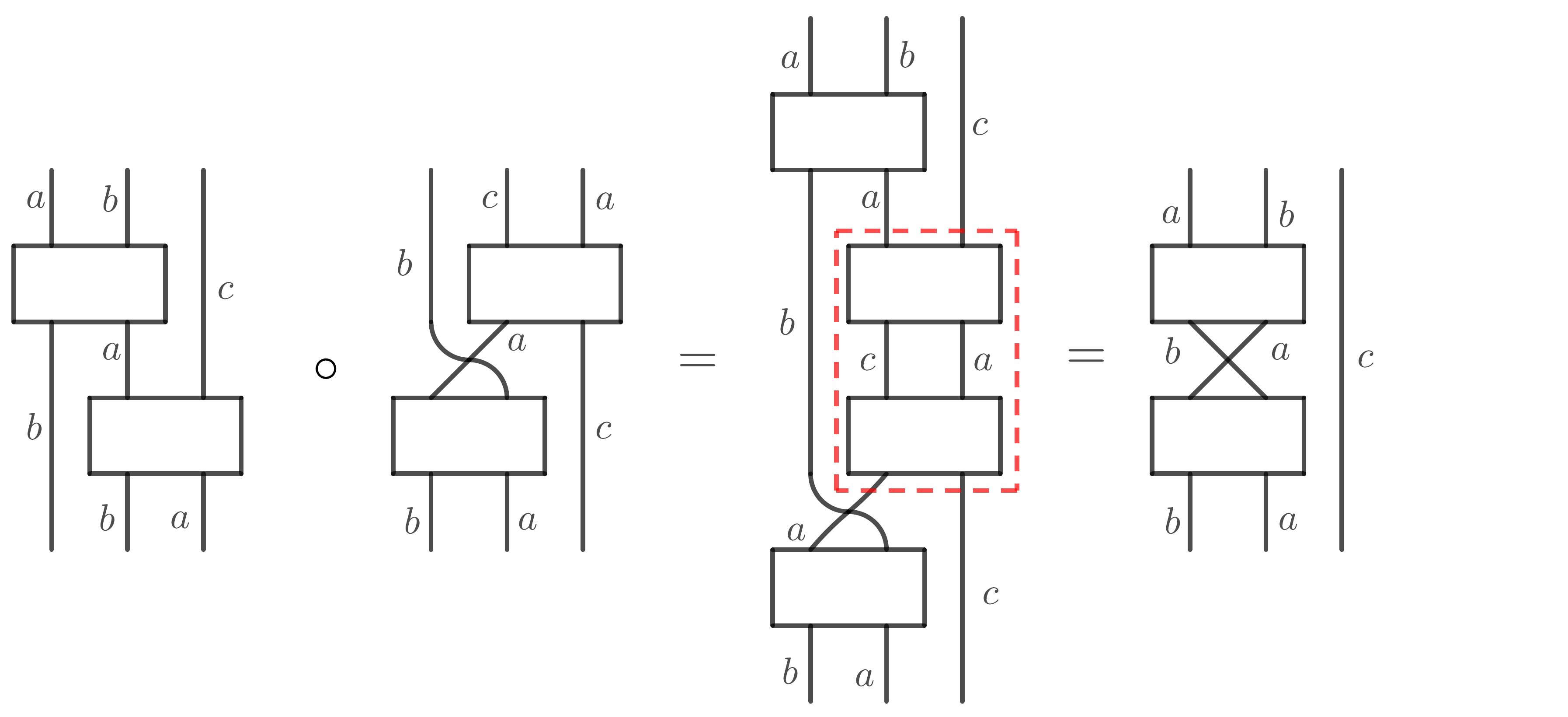}

\caption{}
\label{figure9}
\end{center}
\end{figure}
\end{ex}
\begin{definition}
For every $w \in \Sigma^+$, the \emph{symmetric diagram group} $D_s(\mathcal{P},w)$ is the set of all the symmetric $(w,w)$-symmetric semigroup diagrams over $\mathcal{P}$, modulo dipoles, endowed with the concatenation. 
\end{definition}

\noindent
This is indeed a group according to \cite{MR2136028} who refer them as \textit{picture groups}. They also appear under the name of \textit{braided diagram groups} in the work of Guba-Sapir \cite{MR1396957}.

 \begin{remark}
We recall the remark given in \cite{MR3822286} that the braided diagrams, despite the name, are not truly braided. We noticed that the two braided diagrams are equivalent if there is a certain type of marked homeomorphism between them. Thus, the equivalence does not depend upon the embedding into the larger space. Though these groups seem to have little in common with Artin braid groups. So conventionally, we prefer to call these symmetric diagram groups.
 \end{remark}
 
\begin{ex}\label{ex:V}
The symmetric diagram group $D_s(\mathcal{P},x)$ associated to the semigroup presentation $\mathcal{P} = \langle x \mid x=x^2 \rangle$ is isomorphic to Thompson's group $V$. See \cite[Example 16.6]{MR1396957} for a proof. 
\end{ex}
\medskip

\paragraph{\textbf{Planar and annular diagram groups.}} In this section, we fix a semigroup presentation $\mathcal{P}= \langle \Sigma \mid \mathcal{R} \rangle$ and a baseword $w \in \Sigma^+$. Below, we define two variations of symmetric semigroup diagrams. Our first variation is the main topic of \cite{MR1396957}.

\begin{definition}
A symmetric diagram over $\mathcal{P}$ is \emph{planar} if there exists an embedding $\Delta \to \mathbb{R}^2$ which preserves the left-to-right orderings and the top-bottom orientations on the transistors and the frame. The \emph{planar diagram group} $D_p(\mathcal{P},w)$ is the subgroup of $D_s(\mathcal{P},w)$ consisting of all planar diagrams.
\end{definition}

\begin{ex}\label{ex:F}
The planar diagram group $D_p(\mathcal{P},x)$ associated to the semigroup presentation $\mathcal{P} = \langle x \mid x=x^2 \rangle$ is isomorphic to Thompson's group $F$. See \cite[Example 6.4]{MR1396957} for a proof. 
\end{ex}

\noindent
Our second variation was also introduced in \cite{MR1396957}, and further studied in \cite{MR2136028}. Keep in mind Figure \ref{figure6} when reading the definition.
\begin{figure}
\begin{center}
\includegraphics[trim={0 0cm 29cm 0},clip,scale=0.25]{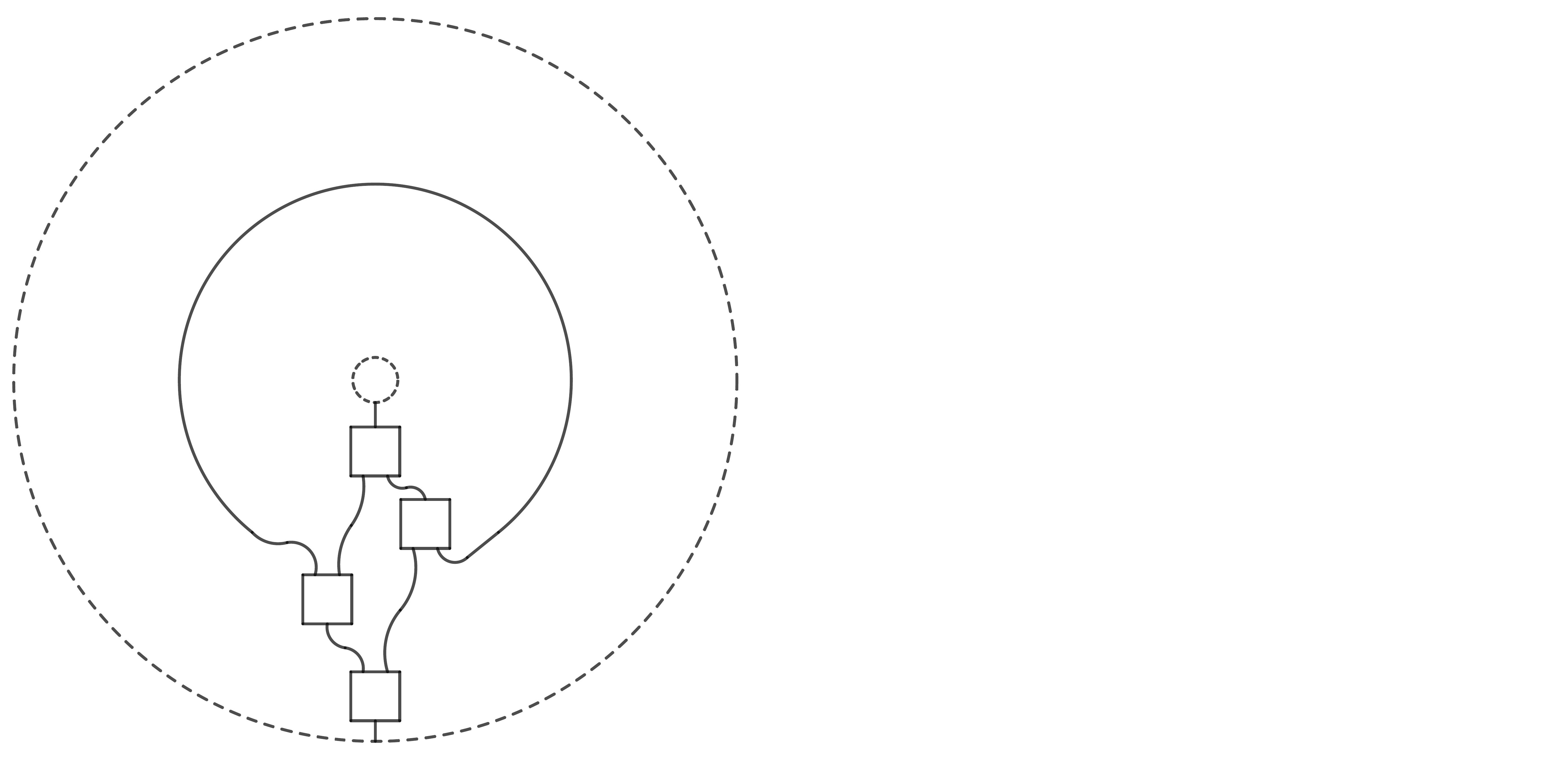}
\caption{Annular diagram.}
\label{figure6}
\end{center}
\end{figure}

\begin{definition}
A symmetric diagram $\Delta$ is \emph{annular} if it can be embedded into an annulus by preserving the left-to-right orderings and the top-bottom orientations on the transistors and the frame. More precisely, suppose that we replace the frame of $\Delta$ with a pair of disjoint circles, both endowed with the counterclockwise orientation in place of the previous left-to-right orderings of the top and bottom sides of the frame; we also fix a basepoint on each circles (which will be disjoint from the wires). Transistors and wires are defined as before, and their attaching maps are subject to the same conditions as before, where the inner (resp. outer) circle of the frame plays the role of the top (resp. bottom) side of the frame. The resulting diagram is \emph{annular} if it embeds into the plane by preserving the left-to-right orderings on the transistors. The \emph{annular diagram group} $D_a(\mathcal{P},w)$ is the set of annular $(w,w)$-diagrams over $\mathcal{P}$, modulo dipoles, endowed with the concatenation.
\end{definition}

\noindent
All the definitions introduced in the previous section naturally generalise to planar and annular semigroup diagrams. Nevertheless, we mention that, when concatenating two annular diagrams, i.e.\ when identifying the top circle of the second diagram with the bottom circle of the second one, we have to match the basepoints on the different circles. 

\begin{ex}\label{ex:T}
The annular diagram group $D_a(\mathcal{P},x)$ associated to the semigroup presentation $\mathcal{P} = \langle x \mid x=x^2 \rangle$ is isomorphic to Thompson's group $T$. See \cite[Example 16.5]{MR1396957} for a proof. 
\end{ex}

\noindent
It is worth noticing that a planar diagram is an annular diagram as well, so that, in the same way that $F \subset T \subset V$, one has
$$D(\mathcal{P},w) \subset D_a(\mathcal{P},w) \subset D_s(\mathcal{P},w)$$
for every semigroup presentation $\mathcal{P}= \langle \Sigma \mid \mathcal{R} \rangle$ and every baseword $w \in \Sigma^+$.

\subsection{Diagram products}

\noindent
In this section, we describe symmetric diagram products as introduced in \cite{MR4033512}. Example~\ref{ex:diagramsPG} and Figure~\ref{figure10} below illustrate the corresponding definitions which we give now. 

\medskip \noindent
Let $\mathcal{P}= \langle \Sigma \mid \mathcal{R} \rangle$ be a semigroup presentation and $\mathcal{G}= \{ G_s \mid s \in \Sigma \}$ a collection of groups indexed by the alphabet $\Sigma$. We set a new alphabet
$$\Sigma(\mathcal{G})= \{ (s,g) \mid s \in \Sigma, \ g \in G_s \}$$
and a new set of relations $\mathcal{R}(\mathcal{G})$ containing
$$ (u_1,g_1) \cdots (u_n,g_n) = (v_1,h_1) \cdots (v_m,h_m) $$
for all $u_1 \cdots u_n=v_1 \cdots v_m \in \mathcal{R}$ and $g_1 \in G_{u_1}, \ldots, g_n \in G_{u_n}$, $h_1 \in G_{v_1}, \ldots, h_n \in G_{u_m}$. We get a new semigroup presentation $\mathcal{P}(\mathcal{G})= \langle \Sigma(\mathcal{G}) \mid \mathcal{R} (\mathcal{G}) \rangle$. A \emph{symmetric diagram over $(\mathcal{P}, \mathcal{G})$} is a symmetric semigroup presentation over $\mathcal{P}(\mathcal{G})$. If $\Delta$ is such a diagram, we denote by $\mathrm{top}^-(\Delta)$ (resp. $\mathrm{bot}^-(\Delta)$) the image of $\mathrm{top}(\Delta)$ (resp. $\mathrm{top}(\Delta)$) under the natural projection $\Sigma(\mathcal{G})^+ \to \Sigma^+$ (which ``forgets'' the second coordinate). If $u,v \in \Sigma^+$ are words, a \emph{$(u,v)$-diagram over $(\mathcal{P}, \mathcal{G})$} is a diagram $\Delta$ satisfying $\mathrm{top}^-(\Delta)=u$ and $\mathrm{bot}^-(\Delta)=v$; we say that $\Delta$ is a \emph{$(u,\ast)$-diagram over $(\mathcal{P}, \mathcal{G})$} if we do not want to mention $v$. 

\medskip \noindent
All the vocabulary introduced in Section \ref{section:diagrams} applies to symmetric diagrams over $(\mathcal{P}, \mathcal{G})$ thought of as symmetric semigroup diagrams over $\mathcal{P}(\mathcal{G})$, except the concatenation and the dipoles which we define now. 

\medskip \noindent
If $\Delta_1$ and $\Delta_2$ are two symmetric diagrams over $(\mathcal{P}, \mathcal{G})$ satisfying $\mathrm{top}^-(\Delta_2)= \mathrm{bot}^-(\Delta_1)$, we define the \emph{concatenation} of $\Delta_1$ and $\Delta_2$ as the symmetric diagram over $(\mathcal{P}, \mathcal{G})$ obtained in the following way. Write $\mathrm{bot}(\Delta_1) = (w_1,g_1) \cdots  (w_n,g_n)$ and $\mathrm{top}(\Delta_2)= (w_1,h_1) \cdots (w_n,h_n)$ for some $w_1, \ldots, w_n \in \Sigma$ and $g_1, \ldots, g_n,h_1, \ldots, h_n \in \bigsqcup\limits_{G \in \mathcal{G}} G$. Now, 
\begin{itemize}
	\item we glue the top endpoints of the wires of $\Delta_2$ connected to the top side of the frame to the bottom endpoints of the wires of $\Delta_1$ connected to the bottom side of the frame (respecting the left-to-right ordering);
	\item we label these wires from left to right by $(w_1,g_1h_1), \ldots, (w_n,g_nh_n)$;
	\item we identify, and then remove, the bottom side of the frame of $\Delta_1$ and the top side of the frame of $\Delta_2$.
\end{itemize}
The symmetric diagram we get is the \emph{concatenation} $\Delta_1 \circ \Delta_2$.

\medskip \noindent
Given a symmetric diagram $\Delta$ over $(\mathcal{P}, \mathcal{G})$, a \emph{dipole} in $\Delta$ is the data of two transistors $T_1,T_2$ satisfying $T_1 \prec T_2$ such that, if $w_1, \ldots, w_n$ denote the wires connected to the top side of $T_1$, listed from left to right:
\begin{itemize}
	\item the top endpoints of the $w_i$'s are connected to the bottom of $T_2$ with the same left-to-right order, and no other wires are attached to the bottom of $T_2$;
	\item the labels $\mathrm{top}^-(T_2)$ and $\mathrm{bot}^-(T_1)$ are the same;
	\item the wires $w_1, \ldots, w_n$ are labeled by letters of $\Sigma(\mathcal{G})$ with trivial second coordinates.
\end{itemize}
Given such a dipole, one may \emph{reduce} it by 
\begin{itemize}
	\item removing the transistors $T_1,T_2$ and the wires $w_1, \ldots, w_n$;
	\item connecting the top endpoints of the wires $a_1, \ldots, a_m$ (from left to right) which are connected to the top side of $T_2$ with the bottom endpoints of the wires $b_1, \ldots, b_m$ (from left to right) which are connected to the bottom side of $T_1$ (preserving the left-to-right orderings);
	\item and labelling the new wires by $( \ell_1, g_1h_1), \ldots, (\ell_m, g_mh_m)$ from left to right, if $a_i$ is labelled by $(\ell_i,h_i)$ and $b_i$ by $(\ell_i,g_i)$ for every $1 \leq i \leq m$.
\end{itemize}
A symmetric diagram over $(\mathcal{P}, \mathcal{G})$ which does not contain any dipole is \emph{reduced}. Of course, any symmetric diagram can be transformed into a reduced one by reducing its dipoles, and the same argument as \cite[Lemma 2.2]{MR2136028} shows that the reduced diagram we get does not depend on the order we choose to reduce its dipoles. We refer to this reduced diagram as the \emph{reduction} of the initial diagram. Two diagrams are the same \emph{modulo dipoles} if they have the same reduction. 

\begin{definition}
For every $w \in \Sigma^+$, the \emph{symmetric diagram product} $D_s(\mathcal{P}, \mathcal{G},w)$ is the set of all the symmetric $(w,w)$-diagrams $\Delta$ over $(\mathcal{P},\mathcal{G})$, modulo dipoles, endowed with the concatenation. 
\end{definition}

\noindent
This is indeed a group, for the same reason that a symmetric diagram group turns out to be a group.

\begin{ex}\label{ex:diagramsPG}
Consider the semigroup presentation $\mathcal{P}= \langle a,b,p \mid a=ap, b=pb \rangle$ and the collection of groups $\mathcal{G}= \{ G_a=G_b=G_p= \mathbb{Z} \}$. Figure \ref{figure10} shows the concatenation of two symmetric diagrams over $(\mathcal{P}, \mathcal{G})$, and the reduction of the resulting diagram. 
\begin{figure}
\begin{center}
\includegraphics[trim={0 4cm 15cm 0},clip,scale=0.3]{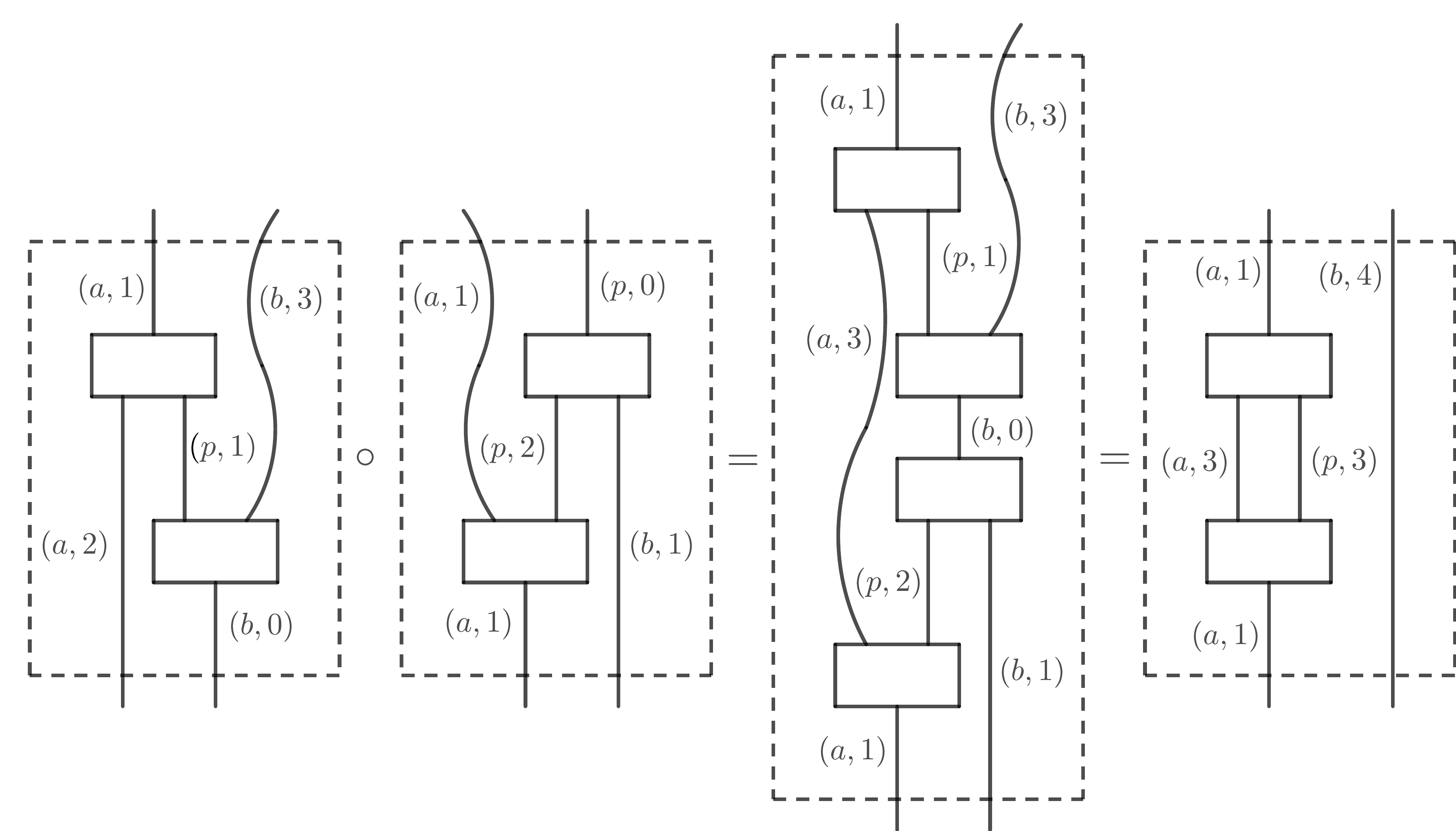}
\caption{Concatenation and reduction of two symmetric diagrams.}
\label{figure10}
\end{center}
\end{figure}
\end{ex}

\subsection{Combination theorem}

\noindent
It is proved in \cite[Theorem~4]{MR1725439} that a planar diagram product of planar diagram groups is again a diagram group. The motivation is clear: given a diagram reprensenting an element of our diagram product, one can produce a diagram over a bigger semigroup presentation by expanding each wire labelled by an element with a diagram representing it. This idea allows one to construct an injective morphism from every (planar or symmetric) diagram product of (planar or symmetric) diagram groups into a (planar or symmetric) diagram group. The difficulty is to choose the latter group in order get a morphism which is also surjective. Let us illustrate what may happen on a specific example.

\medskip \noindent
Let $\mathcal{P} = \langle a \mid \ \rangle$. Given a group $G$, the symmetric diagram product $D(\mathcal{P},\mathcal{G},a^2)$, where $\mathcal{G}= \{G(a):=G\}$, is isomorphic to the wreath product $G \wr \mathbb{Z}/2\mathbb{Z}:= (G \oplus G) \rtimes \mathbb{Z}/2\mathbb{Z}$. Now, assume that $G$ is isomorphic to Thompson's group $V$. We know that $V$ can be described as the symmetric diagram group $D(\mathcal{Q},x)$ where $\mathcal{Q}:= \langle x \mid x=x^2 \rangle$. A natural attempt is to set the semigroup presentation
$$\mathcal{S}:= \langle a,x \mid a=x, x=x^2 \rangle$$
and to ask whether $D(\mathcal{P},\mathcal{G},a^2)$ is isomorphic to $D(\mathcal{S},a^2)$. However, every reduced $(a^2,a^2)$-diagram over $\mathcal{S}$ can be written (up to dipole reduction) as a concatenation $A \circ \Delta \circ A^{-1}$ where $\Delta$ is an $(x^2,x^2)$-diagram over $\mathcal{Q}$ and where $A$ is the $(a^2,x^2)$-diagram given by the derivation $aa \to xa \to xx$. Thus, $D(\mathcal{S},a^2)$ is isomorphic to $D(\mathcal{Q},x^2)$. Since $x=x^2$ modulo $\mathcal{Q}$, the latter is also isomorphic to $D(\mathcal{Q},x)$, and we conclude that $D(\mathcal{S},a^2)$ is actually isomorphic to $V$. 

\medskip \noindent
In our example, the obstruction comes from the fact that the two natural copies of $D(\mathcal{Q},x)$ in $D(\mathcal{S},a^2)$ ``interact'', which is not the case for the two copies of $V$ in the diagram product. In planar diagram products, it is possible to insert ``separation letters'' in order to prevent such an interaction. However, in symmetric diagram products, the trick does not work anymore since wires are allowed to cross.

\medskip \noindent
Actually, it is reasonable to think that, contrary to the planar case, symmetric diagram groups are not stable under symmetric diagram products. As explicit condidates:

\begin{question}
Are the free product $V \ast V$ and the wreath product $V \wr \mathbb{Z}/2\mathbb{Z}$ symmetric diagram groups?
\end{question}

\noindent
Nevertheless, it turns out that symmetric diagram products of symmetric diagram groups are symmetric diagram groups in some cases. Below, we record a simple case which will help us proving Theorem~\ref{thm:MainIntro}. 

\begin{thm}\label{thm:Combination}
Let $\mathcal{P}:= \langle \Sigma \mid \mathcal{R} \rangle$ be a semigroup presentation and $w \in \Sigma^+$ a word. Set $\mathcal{G}:= \{ G(s)= \mathbb{Z}, s \in \Sigma\}$. The diagram product $D(\mathcal{P},\mathcal{G},w)$ is isomorphic to the diagram group $D(\mathcal{Q},w)$ where
$$\mathcal{Q}:=\left\langle \Sigma \sqcup \bigsqcup_{s \in \Sigma} \{a_s,b_s,c_s\} \mid \mathcal{R} \sqcup \bigsqcup\limits_{s \in \Sigma} \{s=a_s,a_s=b_s,b_s=c_s, c_s=s\} \right\rangle.$$
\end{thm}

\begin{proof}
To every diagram $\Delta$ over $(\mathcal{P},\mathcal{G})$, we associate a diagram $\Theta(\Delta)$ over $\mathcal{Q}$ as follows. Given a wire in $\Delta$ labelled by some $(s,k)$ with $s \in \Sigma$ and $k \in G(s)$, we replace it with the $(s,s)$-diagram $M_s(k)$ over $\mathcal{Q}$ given by the derivation
$$s \to \underset{k \text{ times}}{\underbrace{a_s \to b_s \to c_s \to \cdots \to a_s \to b_s \to c_s}} \to s $$
if $k \geq 0$ or its inverse if $k<0$. 

\medskip \noindent
The first observation is that $\Theta$ sends diagrams which are equivalent modulo dipole reduction to diagrams which are equivalent modulo dipole reduction. Indeed, let $\Delta'$ be a diagram obtained from another diagram $\Delta$ by reducing a dipole, say $(T_1,T_2)$. The wires connecting $T_1$ and $T_2$ do not cross and have trivial second coordinates. Therefore, $(T_1,T_2)$ also defines a dipole in $\Theta(\Delta)$. Clearly, reducing $(T_1,T_2)$ in $\Theta(\Delta)$ yields $\Theta(\Delta')$. This justifies our assertion.

\medskip \noindent
Thus, $\Theta$ defines a map $D(\mathcal{P},\mathcal{G},w) \to D(\mathcal{Q},w)$. Let us verify that it is a morphism. Let $\Phi,\Psi \in D(\mathcal{P},\mathcal{G},w)$ be two diagrams and let $w_1, \ldots,w_n$ denote the letters of $w$ from left to right. We can decompose $\Phi$ (resp. $\Psi$) as a concatenation $\Phi_0 \circ \epsilon$ (resp. $\eta \circ \Phi_0$) where $\Phi_0$ (resp. $\Psi_0$) has its bottom (resp. top) wires whose labels have trivial second coordinates and where $\epsilon$ (resp. $\eta$) has no transistor but only wires labelled from left to right by, say, $g_1 \in G(w_1), \ldots, g_n \in G(w_n)$ (resp. $h_1 \in G(w_1), \ldots, h_n \in G(w_n)$). Then $\Phi \circ \Psi$ can be written as $\Phi_0 \circ \mu \circ \Psi_0$ where $\mu$ has no transistor by only wires labelled from left to right by $g_1h_1, \ldots, g_nh_n$. By construction, $\Theta( \Phi \circ \Psi)= \Theta( \Phi_0) \circ \Theta(\mu) \circ \Theta(\Psi_0)$. But we clearly have $\Theta(\mu) \equiv \Theta(\epsilon) \circ \Theta(\eta)$, $\Theta(\Phi_0) \circ \Theta(\epsilon)$, and $\Theta(\eta) \circ \Theta(\Psi_0)= \Theta(\Psi)$. Hence $\Theta(\Phi \circ \Psi) \equiv \Theta(\Phi) \circ \Theta(\Psi)$. 

\medskip \noindent
Now, let us verify that $\Theta$ is injective. Let $\Delta \in D(\mathcal{P},\mathcal{G},w)$ be a diagram. Assume that $\Theta(\Delta)$ contains a dipole, say $(T_1,T_2)$. Because $\Theta(\Delta)$ is obtained from $\Delta$ by expanding the wires with reduced diagrams, necessarily $T_1$ and $T_2$ already existed in $\Delta$, and necessarily define a dipole in $\Delta$. Therefore, $\Theta$ sends reduced diagrams to reduced diagrams, which implies its injectivity.

\medskip \noindent
Finally, it remains to show that $\Theta$ is surjective. So let $\Delta \in D(\mathcal{Q},w)$ be an arbitrary reduced diagram. Because $\Delta$ is reduced, for every $s \in \Sigma$:
\begin{itemize}
	\item a transistor labelled by $a_s=b_s$ must lie between two transistors labelled by $s=a_s$ and $b_s=c_s$ or $c_s=a_s$ and $b_s=a_s$; 
	\item a transistor labelled by $b_s=c_s$ must lie between two transistors labelled by $a_s=b_s$ and $c_s=s$ or $c_s=a_s$;
	\item a transistor labelled by $c_s=a_s$ must lie between two transistors labelled by $b_s=c_s$ and $a_s=s$ or $a_s=b_s$.
\end{itemize}
As a consequence, there exist subdiagrams $\Phi_1, \ldots, \Phi_n$ such that every transistor labelled by a relation not in $\mathcal{R}$ is contained in some $\Phi_i$ and such that, for each $1 \leq i \leq n$, $\Phi_i = M_s(k)$ for some $s \in \Sigma$ and $k \in \mathbb{Z}$. Let $\Psi$ denote the diagram obtained from $\Delta$ by collapsing each such $\Phi_i$ to a single wire labelled by $(s,k)$. Then $\Psi$ is a diagram over $(\mathcal{P},w)$ and $\Theta(\Psi)=\Delta$ by construction. 
\end{proof}

\section{Graph products}

\subsection{Preliminaries}

\noindent
Let $\Gamma$ be a simplicial graph and $\mathcal{G}= \{ G_u \mid u \in V(\Gamma) \}$ be a collection of groups indexed by the vertex-set $V(\Gamma)$ of $\Gamma$. The \emph{graph product} $\Gamma \mathcal{G}$ is defined as the quotient
$$\left( \underset{u \in V(\Gamma)}{\ast} G_u \right) / \langle \langle [g,h]=1, g \in G_u, h \in G_v \ \text{if} \ (u,v) \in E(\Gamma) \rangle \rangle$$
where $E(\Gamma)$ denotes the edge-set of $\Gamma$. The groups of $\mathcal{G}$ are referred to as \emph{vertex-groups}. 

\medskip \noindent
A \emph{word} in $\Gamma \mathcal{G}$ is a product $g_1 \cdots g_n$ for some $n \geq 0$ and, for every $1 \leq i \leq n$, $g_i \in G$ for some $G \in \mathcal{G}$; the $g_i$'s are the \emph{syllables} of the word, and $n$ is the \emph{length} of the word. Clearly, the following operations on a word does not modify the element of $\Gamma \mathcal{G}$ it represents:
\begin{description}
	\item[Cancellation] delete the syllable $g_i=1$;
	\item[Amalgamation] if $g_i,g_{i+1} \in G$ for some $G \in \mathcal{G}$, replace the two syllables $g_i$ and $g_{i+1}$ by the single syllable $g_ig_{i+1} \in G$;
	\item[Shuffling] if $g_i$ and $g_{i+1}$ belong to two adjacent vertex-groups, switch them.
\end{description}
A word is \emph{graphically reduced} if its length cannot be shortened by applying these elementary moves. Every element of $\Gamma \mathcal{G}$ can be represented by a graphically reduced word, and this word is unique up to the shuffling operation. For more information about graphically reduced words, we refer to \cite{GreenGP} (see also \cite{HsuWise,VanKampenGP}).

\subsection{Main theorem}The rest of the section is dedicated to the proof of Theorem~\ref{thm:GraphProduct}. The strategy is to realise graph products of groups as symmetric diagram groups and next to apply Theorem~\ref{thm:Combination} in order to conclude.

\begin{definition}
Let $n \in \mathbb{N} \cup \{\infty\}$ be an integer and $\mathscr{C}$ a collection of subsets in $[n]:= \{1, \ldots, n\}$. The \emph{disjointness graph} $\Delta_\mathscr{C}$ of $\mathscr{C}$ is the graph whose vertex-set is $\mathscr{C}$ and whose edges connect two elements of $\mathscr{C}$ whenever they are disjoint. 
\end{definition}

\noindent
As a preliminary observation, we show that every countable graph can be realised as a disjointness graph.

\begin{lemma}\label{lem:DisjointGraph}
For every countable graph $\Gamma$, there exist some $n \in \mathbb{N} \cup \{\infty\}$ and some collection of finite subsets $\mathscr{C}$ of $[n]$ such that $\Delta_\mathscr{C}$ is isomorphic to $\Gamma$.
\end{lemma}

\begin{proof}
If $\Gamma$ is the union of two isolated vertices, it suffices to take $n=2$ and $\mathscr{C}=\{ \{1\}, \{1,2\}\}$. From now on, we assume that $\Gamma$ is not the union of two isolated vertices.

\medskip \noindent
Let $\Gamma^{\mathrm{opp}}$ denote the opposite graph of $\Gamma$, namely the graph whose vertex-set is $V(\Gamma)$ and whose edges connect two vertices whenever they are not adjacent in $\Gamma$. Let $n \in \mathbb{N} \cup \{\infty\}$ denote the number of edges in $\Gamma^{\mathrm{opp}}$. From now on, we identify $[n]$ with the edge-set of $\Gamma^{\mathrm{opp}}$. For every vertex $u$ of $\Gamma$, let $S(u) \subset [n]$ denote the set of the edges of $\Gamma^\mathrm{opp}$ containing $u$. 

\medskip \noindent
We claim that the map $\Psi : u \mapsto S(u)$ induces a graph isomorphism $\Gamma \to \Delta_\mathscr{C}$ where $\mathscr{C}:= \{ S(u), u \in V(\Gamma)\}$. Because $\Gamma$ is the union of two isolated vertices, $\Psi$ induces a bijection between the vertices of $\Gamma$ and $\Delta_\mathscr{C}$. If two vertices $u,v \in \Gamma$ are adjacent, then there is no edge in $\Gamma^\mathrm{opp}$ connecting $u$ and $v$, which amounts to saying that $S(u)$ and $S(v)$ are disjoint. Thus, $S(u)$ and $S(v)$ are adjacent in $\Delta_\mathscr{C}$. Conversely, if $u,v \in \Gamma$ are not adjacent, then $S(u)$ and $S(v)$ contain the edge of $\Gamma^\mathrm{opp}$ connecting $u$ and $v$, so $S(u)$ and $S(v)$ cannot be adjacent in $\Delta_\mathscr{C}$. 
\end{proof}

\noindent

\begin{thm}\label{thm:GraphProduct}
Let $n \in \mathbb{N}$ be an integer, $\mathscr{C}$ a collection of subsets in $[n]$, and $\mathcal{G}:= \{G_I, I \in \mathscr{C}\}$ a collection of groups indexed by $\mathscr{C}$. Let $\mathcal{P}_\mathscr{C}$ denote the semigroup presentation
$$\left\langle x_i \ (1 \leq i \leq n), \ a_I \ (I \in \mathscr{C}) \mid x_{i_1} \cdots x_{i_k} = a_I \ (I=\{i_1< \cdots < i_k\} \in \mathscr{C}) \right\rangle.$$
Setting $\mathcal{H}:= \{ H(x_i)= \{1\} \ (1 \leq i \leq n), \ H(a_I)=G_I \ (I \in \mathscr{C})\}$, the symmetric diagram product $D_s(\mathcal{P}_\mathscr{C},\mathcal{H},x_1 \cdots x_n)$ is isomorphic to the graph product $\Delta_\mathscr{C} \mathcal{G}$. 
\end{thm}

\begin{proof}
For convience, set $w:=x_1 \cdots x_n$. For every vertex $I= \{x_{i_1}< \cdots < x_{i_k}\} \in \Delta_\mathscr{C}$ and every element $g \in G_I$, let $\Theta(g)$ denote the $(w,w)$-diagram obtained by connecting the wires $x_{i_1}, \ldots, x_{i_k}$ to a transistor $T$ labelled by $x_{i_1} \cdots x_{i_k} =a_I$, by labelling the bottom wire of $T$ by $(a_I,g)$, which we connect to a transistor labelled by $a_I=x_{i_1} \cdots x_{i_k}$. See Figure~\ref{Theta}. 
\begin{figure}[h!]
\begin{center}
\includegraphics[width=0.8\linewidth]{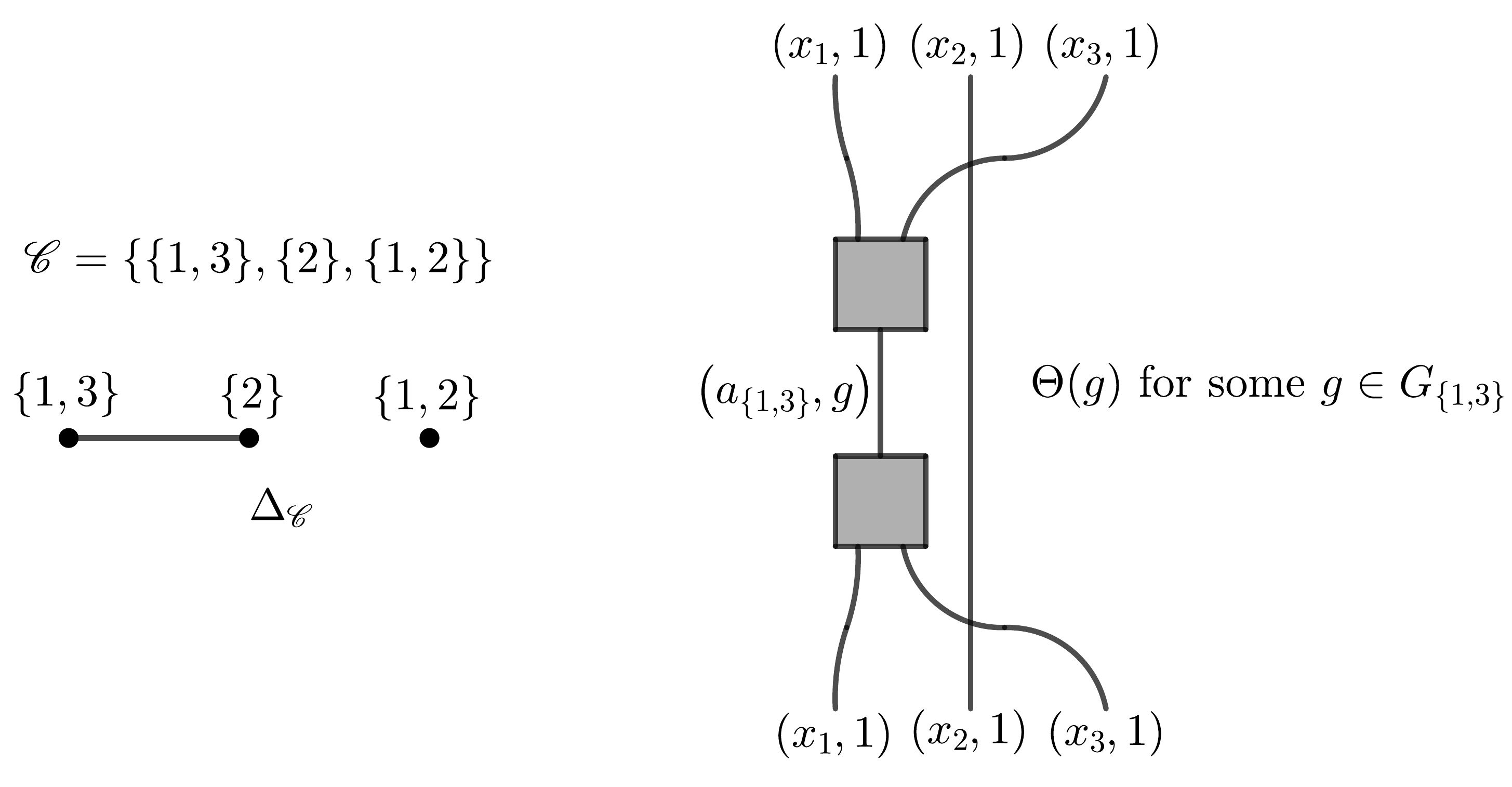}
\caption{Example of a $\Theta(g)$.}
\label{Theta}
\end{center}
\end{figure}

\noindent
Observe that:
\begin{itemize}
	\item for every $I \in \Delta_\mathscr{C}$ and for all $g,h \in G_I$, $\Theta(g) \circ \Theta(h) \equiv \Theta(gh)$;
	\item for all adjacent vertices $I,J \in \Delta_\mathscr{C}$ and all elements $g \in G_I$, $h \in G_J$, we have $\Theta(g) \circ \Theta(h) = \Theta(h) \circ \Theta(g)$ because the transistors of the two diagrams are connected to two disjoint sets of wires (respectively labelled by $x_i$ for $i \in I$ and $x_j$ for $j \in J$). 
\end{itemize}
It follows that $\Theta$ extends to a morphism $\Delta_\mathscr{C} \mathcal{G} \to D(\mathcal{P}_\mathscr{C}, \mathcal{H},w)$ by sending a word in generators $g_1 \cdots g_n$ to $\Theta(g_1) \circ \cdots \circ \Theta(g_n)$.

\medskip \noindent
Fix a word in generators $g=g_1 \cdots g_n$. If $\Theta(g)=\Theta(g_1) \circ \cdots \circ \Theta(g_n)$ contains a dipole, two cases may occur. First, there may exist some $1 \leq i \leq n$ such that $\Theta(g_i)$ contains a dipole, which amounts to saying that $g_i=1$. Otherwise, if all the $\Theta(g_i)$ are reduced, then there must exist $1 \leq i < j \leq n$ such that our dipole contains one transistor in $\Theta(g_i)$ and the other in $\Theta(g_j)$. This implies that $g_i$ and $g_j$ belong to the same vertex-group of $\Delta_\mathscr{C}\mathcal{G}$, say $G_I$ with $I \in \mathscr{C}$, and that the \emph{captive} wires (i.e.\ the top and bottom wires connected to a transistor) of the $\Theta(g_\ell)$ for $i < \ell < j$ are connected to the \emph{free} (i.e.\ not captive) wires of $\Theta(g_i)$ and $\Theta(g_j)$. This amounts to saying that each $g_\ell$ for $i<\ell<j$ belongs to a vertex-group $G_L$ with $L$ disjoint from $I$. Thus, in our word $g_1 \cdots g_n$, the syllable $g_j$ can be shuffled next to $g_i$ and finally merged with $g_j$. 

\medskip \noindent
Our argument shows that, if $\Theta(g)$ is not reduced, then $g_1 \cdots g_n$ is not graphically reduced. Consequently, $\Theta$ sends a (non-empty) graphically reduced word in generators to a (non-trivial) reduced diagram. This implies that $\Theta$ is an injective morphism. 

\medskip \noindent
It remains to show that $\Theta$ is surjective. So let $\Delta \in D(\mathcal{P}_\mathscr{C},\mathcal{H},w)$ be a reduced diagram. If $\Delta$ is trivial, then $\Theta(1)= \Delta$. From now on, we assume that $\Delta$ is non-trivial. Necessarily, $\Delta$ contains a transistor $R$ labelled by $x_{i_1} \cdots x_{i_k}=a_I$ for some $I=\{i_1< \cdots < i_k\} \in \mathscr{C}$ such that the top wires of $R$ are also top wires of $\Delta$. The bottom wire of $R$, whose label has first coordinate $a_I$, must be connected to a transistor $S$ labelled by $a_I=x_{i_1} \cdots x_{i_k}$. Thus, the union of $R$ and $S$ coincides with $\Theta(g)$ where $g \in G_I$ is such that the wire connecting $R$ and $S$ is labelled by $(a_I,g)$. Then, the reduction of $\Theta(g)^{-1} \circ \Delta$ has less transistors than $\Delta$. By iterating the argument, we conclude that $\Delta$ belongs to the image of $\Theta$.
\end{proof}

\begin{proof}[Proof of Theorem~\ref{thm:MainIntro}.]
The theorem now follows from Theorem~\ref{thm:GraphProduct}, Lemma~\ref{lem:DisjointGraph}, and Theorem~\ref{thm:Combination}. 
\end{proof}

\section{The pure virtual twin group is a braided diagram group}

The purpose of this section is to give a new perspective to the pure virtual twin group by viewing them as symmetric diagram groups. We recall the preliminaries in the subsequent section.

\subsection{Virtual twin groups}
The {\it virtual twin group} $VT_n$, $n \ge 2$, is generated by the set $\{ s_1, s_2, \ldots, s_{n-1}, \rho_1, \rho_2, \ldots, \rho_{n-1}\}$ with defining relations
\begin{eqnarray}
s_i^{2} &=&1 \hspace*{5mm} \textrm{for } i = 1, 2, \dots, n-1, \label{1}\\ 
s_is_j &=& s_js_i \hspace*{5mm} \textrm{for } |i - j| \geq 2,\label{2}\\
\rho_i^{2} &=& 1 \hspace*{5mm} \textrm{for } i = 1, 2, \dots, n-1, \label{3}\\
\rho_i\rho_j &=& \rho_j\rho_i \hspace*{5mm} \textrm{for } |i - j| \geq 2, \label{4}\\
\rho_i\rho_{i+1}\rho_i &=& \rho_{i+1}\rho_i\rho_{i+1}\hspace*{5mm} \textrm{for } i = 1, 2, \dots, n-2, \label{5}\\
\rho_is_j &=& s_j\rho_i \hspace*{5mm} \textrm{for } |i - j| \geq 2, \label{6}\\
\rho_i\rho_{i+1} s_i &=& s_{i+1} \rho_i \rho_{i+1}\hspace*{5mm} \textrm{for } i = 1, 2, \dots, n-2. \label{7}
\end{eqnarray}

In particular,  $VT_2 \cong \mathbb{Z}_2 * \mathbb{Z}_2$, the infinite dihedral group. There is a natural surjection $\pi:VT_n  \to S_n$ given by
$$\pi(s_i) = \pi(\rho_i) = (i, i+1)$$
for all $1\leq i \leq n-1$. The kernel $PVT_n$ of this surjection is called the \textit{pure virtual twin group}. The group $PVT_n$ is an analogue of the pure virtual braid group. The map $S_n \to VT_n$ given by $(i, i+1)\mapsto \rho_i$ is a splitting of the short exact sequence
$$1 \to PVT_n \to VT_n \to S_n \to 1,$$
and hence $VT_n= PVT_n \rtimes S_n$. Let 
$$\lambda_{i, i+1}= s_i \rho_i,$$
for each $1 \le i \le n-1$ and
$$\lambda_{i,j} = \rho_{j-1} \rho_{j-2} \dots \rho_{i+1} \lambda_{i, i+1} \rho_{i+1} \dots \rho_{j-2}  \rho_{j-1},$$
for each $1 \leq i < j \leq n$ and $j \ne i+1$.
The following result \cite{Structure-Automorphisms-Pure-Virtual-Twin-Groups} describes the structure of $PVT_n$.
\begin{thm}\label{pvtn-right-angled-artin}
The pure virtual twin group $PVT_n$ on  $n \ge 2 $ strands is an irreducible right-angled Artin group and is presented by
$$\big\langle \lambda_{i,j},~1 \leq i < j \leq n ~|~ \lambda_{i,j} \lambda_{k,l} =  \lambda_{k,l} \lambda_{i,j} \text{ for distinct integers } i, j, k, l \big\rangle.$$
\end{thm}
These virtual twin groups are classical generalisation of twin groups, which are generated by the set $\{ s_1, s_2, \ldots, s_{n-1}\}$ whose generators satisfy the defining relations (\ref{1})--(\ref{2}), whose pure subgroup is denoted by $PT_n$.
Analogous to (virtual) braid groups \cite{MR2128049}, (virtual) twin groups have a nice diagrammatical interpretation which we describe in this section (see \cite{MR4209535} for more details).\\
Consider a set $Q$ of fixed $n$ points on the real line $\mathbb{R}$. A \textit{virtual twin diagram} on $n$ strands is a subset $D$ of the strip $\mathbb{R} \times [0,1]$ consisting of $n$ intervals called {\it strands} such that the boundary of $D$ is $Q  \times \{0,1\}$ and the following conditions are met:
\begin{enumerate}
\item the natural projection $\mathbb{R} \times [0,1] \to [0,1]$ maps each strand homeomorphically onto $[0,1]$. Informally, the strands are monotonic,
\item the set $V(D)$ of all crossings of the diagram $D$ consists of transverse double points of $D$ where each crossing is preassigned to be a real or a virtual crossing as shown below. A virtual crossing is depicted by a crossing encircled with a small circle.
\end{enumerate}

\begin{figure}[hbtp]
\centering
\includegraphics[scale=0.4]{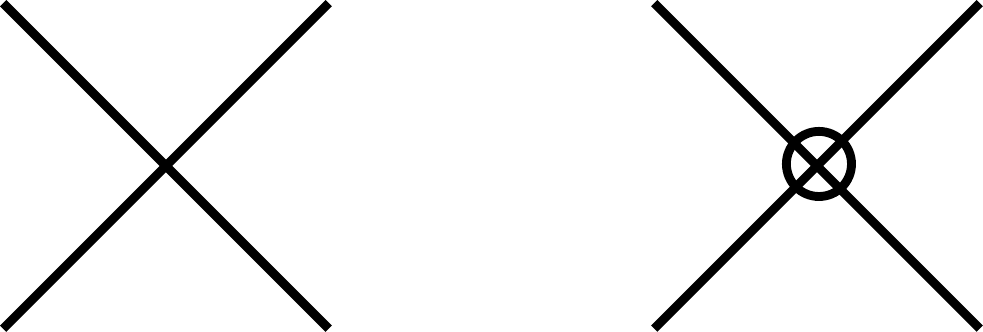}
\caption{Real and virtual crossing}
\label{Crossings}
\end{figure}

We say that two virtual twin diagrams on $n$ strands are said to be \textit{equivalent} if one can be obtained from the other by isotopies of the plane and a finite sequence of planar Reidemeister moves as shown in Figure [\ref{ReidemeisterMoves}]. 
\begin{figure}[hbtp]
\centering
\includegraphics[scale=0.2]{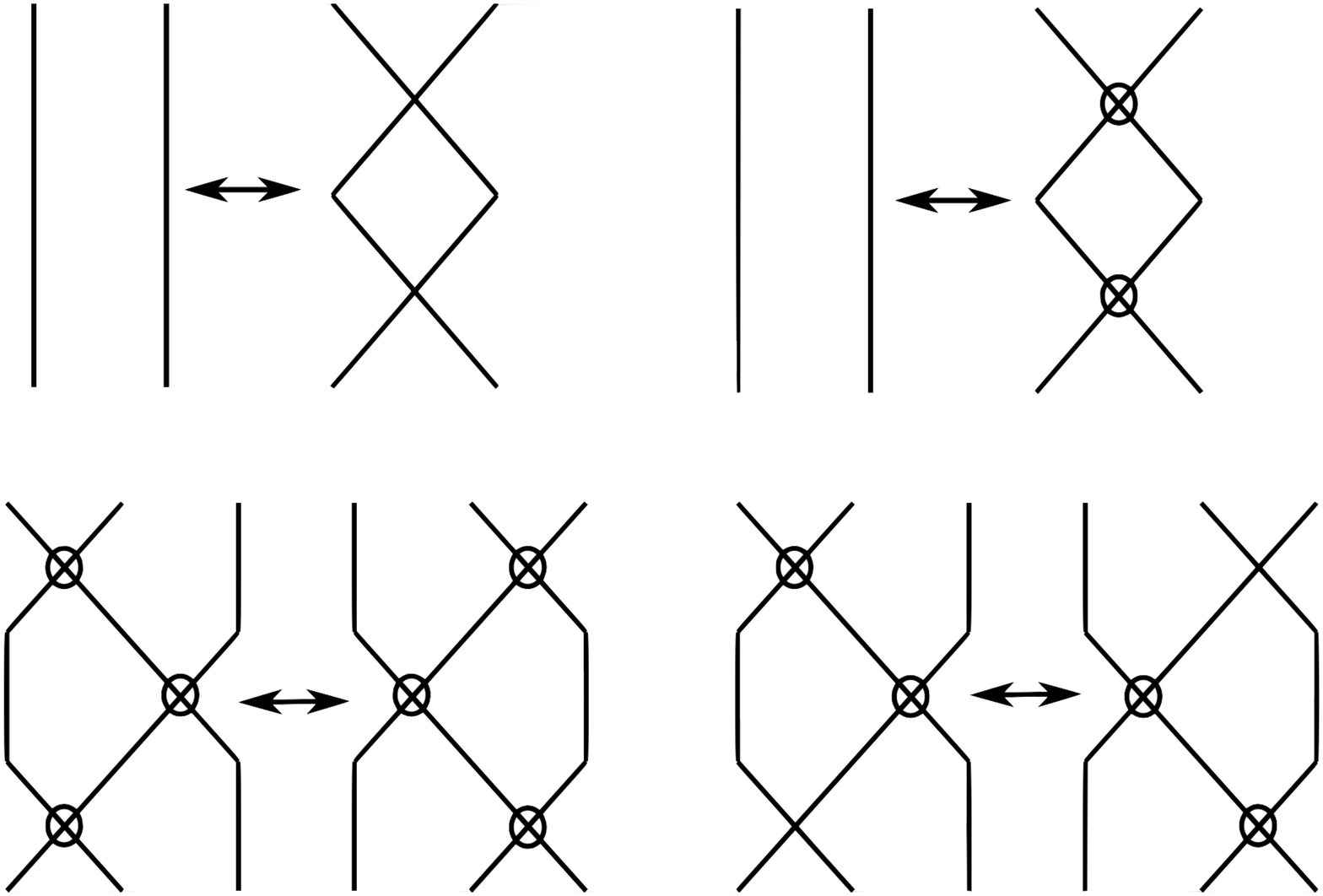}
\caption{Reidemeister moves for virtual twin diagrams}
\label{ReidemeisterMoves}
\end{figure}

A \textit{virtual twin} is then defined as an equivalence class of such virtual twin diagrams. The product $D_1D_2$ of two virtual twin diagrams $D_1$ and $D_2$ is defined by placing $D_1$ on top of $D_2$ and then shrinking the interval to $[0,1]$. It is not difficult to check that this is a well-defined binary operation on the set of all virtual twins on $n$ strands. This set of all virtual twins on $n$ strands forms a group which is isomorphic to the abstractly defined group $VT_n$. The generators $s_i$ and $\rho_i$ of $VT_n$ can be represented as in Figure [\ref{generator-vtn}]. This approach of defining virtual twin groups is crucial in viewing pure virtual twin group as a symmetric diagram group, which we prove in the following subsection.

\begin{figure}[hbtp]
\centering
\includegraphics[scale=0.2]{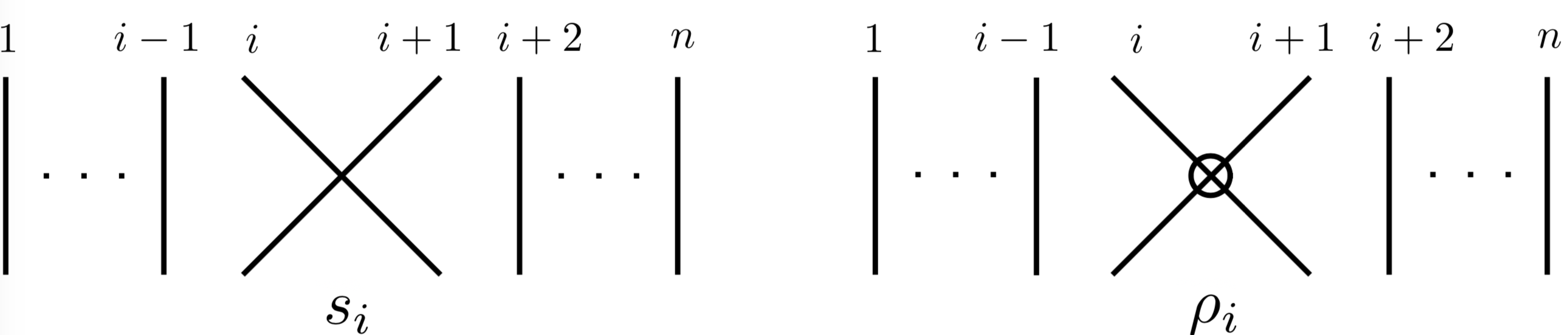}
\caption{Generator $s_i$ and $\rho_i$}
\label{generator-vtn}
\end{figure}

\subsection{Main theorems}
We first start by proving that the group $PVT_n$ is not a planar diagram group. 

\begin{thm}\label{thm:NotPlanar}
The group $PVT_n$ is a planar diagram group if and only if $n \leq 4$.
\end{thm}

\begin{proof}
It follows from Theorem~\ref{pvtn-right-angled-artin} that $PVT_2 \simeq \mathbb{Z}$, $PVT_3 \simeq \mathbb{F}_3$, and $PVT_4 \simeq \mathbb{Z}^2 \ast \mathbb{Z}^2 \ast \mathbb{Z}^2$. All these groups are planar diagram groups. Now, fix an integer $n \geq 5$. According to Theorem~\ref{pvtn-right-angled-artin}, $PVT_n$ is isomorphic to the right-angled Artin group $A(\Gamma_n)$, where $\Gamma_n$ is the graph whose vertices are the pairs of integers $\{i<j\}$ in $[1,n]$ and whose edges connect two pairs whenever they are disjoint. As mentioned in \cite[Corollary~3.19]{Diagram-Groups-Genevois}, the proof of \cite[Theorem~30]{MR1725439} shows that a right-angled Artin group whose defining graph contains an induced cycle of odd length $\geq 5$ cannot be a planar diagram group. Therefore, it suffices to exhibit in $\Gamma_n$ an induced cycle of length $5$. The vertices $\{1,2\}$, $\{2,3\}$, $\{3,5\}$, $\{4,5\}$, $\{1,4\}$ define such a cycle.
\end{proof}

\begin{thm}
Let $\mathcal{P}_n= \langle x_1, x_2, \dots, x_n ~|~ x_i x_j =x_j x_i \text{ for all } i \neq j \rangle$ be a semigroup presentation. The symmetric diagram group $D_s(\mathcal{P}_n,x_1x_2 \cdots x_n )$ is isomorphic to the pure virtual twin group $PVT_n$. 
\end{thm}

\begin{proof}
Consider a pure virtual twin $b$, we stack the twin in a frame such that the strings are labelled $x_1, x_2, \dots , x_n$ from left to right, and the ends of the strings connect the top of the frame with the bottom of the frame. With abuse of notation, we denote the top contacts (also bottom contacts) of frame, reading from the left to right, by $x_1, x_2, \dots , x_n$. Since $b \in PVT_n$, this configuration is a well-defined operation as shown in figure below. 

\begin{figure}[hbtp]
\centering
\includegraphics[scale=0.15]{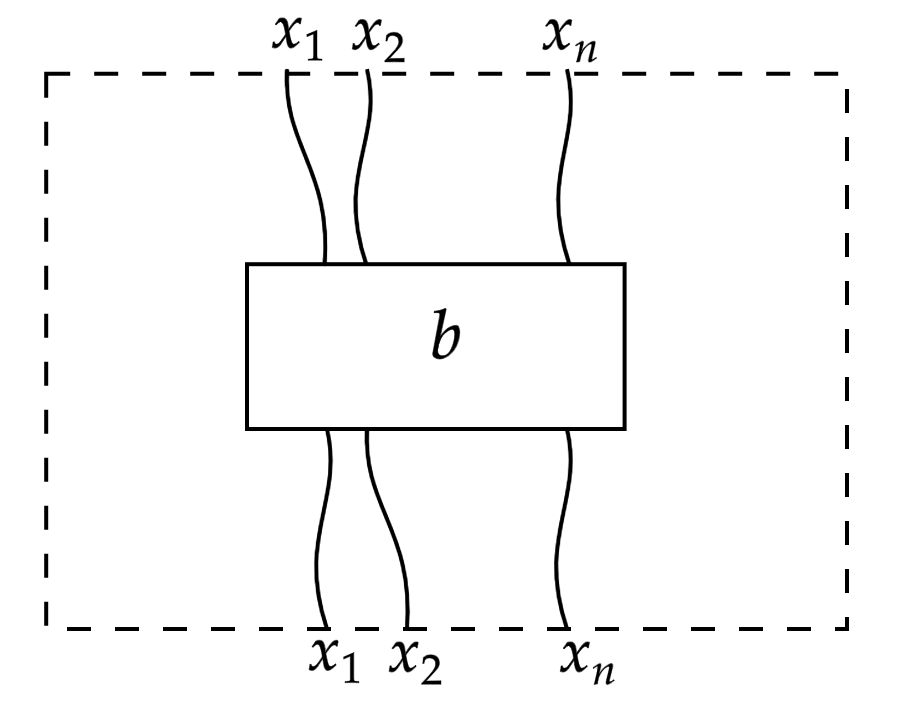}
\end{figure}

In the figure of $b$, there are real and virtual crossings, which we assume to be transversal, we will consider the virtual crossings to be the crossings amongst the wires (without loss of generality, we may assume that) and use the real crossings to be used to construct transistors. Consider a real crossing $c$ in which the $x_i$ strand crosses the strand $x_j$ from the left as we go down the twin diagram. Then we replace the crossing $c$ with $(x_ix_j, x_jx_i)$-transistor. It is to be noted here that the real crossings correspond to the transistors and the virtual crossings correspond to the intersection of wires in the symmetric diagram. For clarity of the proof, we encircle the intersection of wires. Doing this for all the real crossings, we get a $(x_1x_2 \cdots x_n, x_1x_2 \cdots x_n)$-symmetric diagram over the presentation $\mathcal{P}_n$. Conversely, consider any $(x_1x_2 \cdots x_n, x_1x_2 \cdots x_n)$-symmetric diagram over the presentation $\mathcal{P}_n$, then replace all the transistors with the real crossings, and we get a pure virtual twin diagram. \\
We now show that the map between $PVT_n$ and $D_s(\mathcal{P}_n,x_1x_2 \cdots x_n )$ is a bijection. The correspondence between the equivalence relations between twin diagrams and symmetric diagrams are shown in Figure \ref{fig:Proof-Illustration-II}.
\begin{figure}[hbtp]
\centering
\includegraphics[scale=0.1]{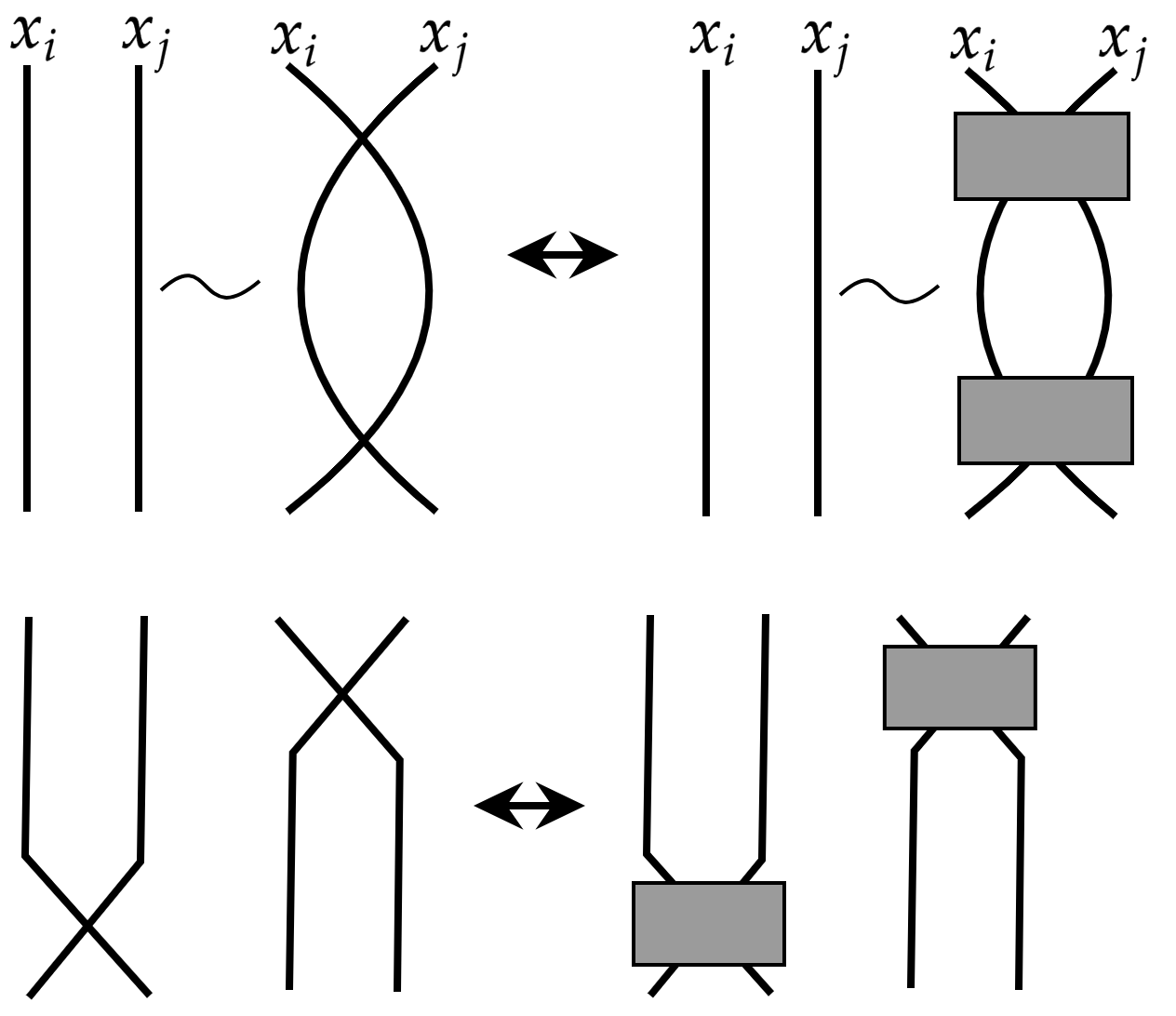}
\caption{}
\label{fig:Proof-Illustration-II}
\end{figure}

\begin{figure}[hbtp]
\centering
\includegraphics[scale=0.14]{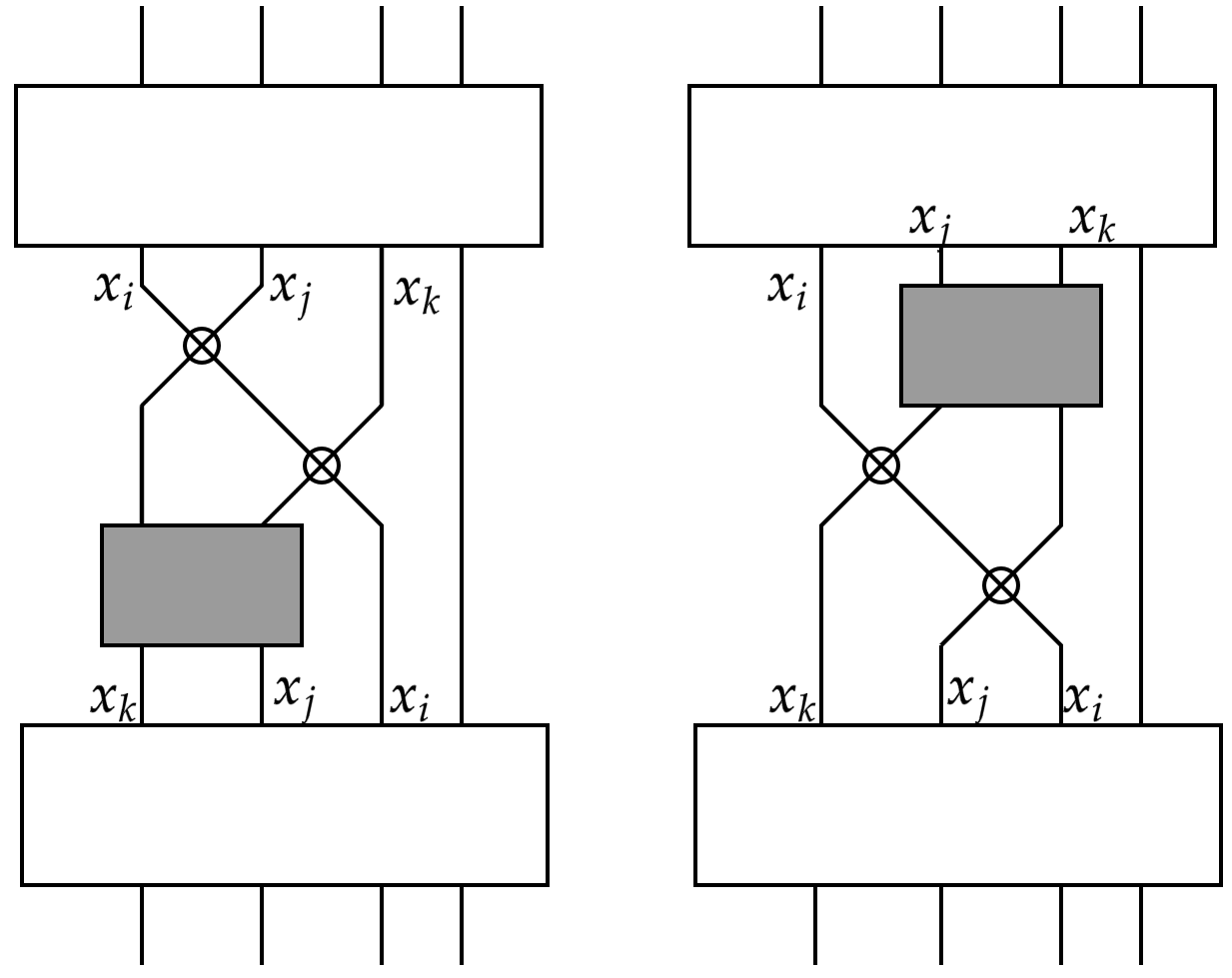}
\caption{}
\label{fig:Proof-Illustration}
\end{figure}

This implies that the real relations in twin diagrams are in correspondence with the equivalent modulo dipoles. Next, it is easy to that the virtual moves in twin diagrams correspond to the isotopy of wires. One of the moves is shown in Figure \ref{fig:Proof-Illustration}. It is crucial to remember that the definition allows that by isotopy a wire can intersect a transistor, since it does not alter the top and bottom contacts of the transistor. Thus, this map is indeed an isomorphism, and we are done. 
\end{proof}

We refer the readers \cite[Section 1.3]{Mostovoy-Round-Twin} definition of annular twin group and we get the following result.

\begin{cor}
The pure twin group embeds in the pure annular twin group which further embeds in the pure virtual twin group.
\end{cor}

\begin{proof}
It follows from the fact that the pure (annular) twin groups are (annular) planar diagram groups \cite{Twin-Group-Diagram-Group}, and that for semigroup presentation $\mathcal{P}_n= \langle x_1, x_2, \dots, x_n ~|~ x_i x_j =x_j x_i \text{ for all } i \neq j \rangle$ and word $w=x_1x_2 \cdots x_n \in \Sigma^+$, we have $D_p(\mathcal{P},w) \subset D_a(\mathcal{P},w) \subset D_s(\mathcal{P},w)$.
\end{proof}

\addcontentsline{toc}{section}{References}

\bibliographystyle{alpha}
{\footnotesize\bibliography{RAAGsDiag}}

\Address

%

\end{document}